\documentclass[10pt]{amsart}
\usepackage[margin=1.5in]{geometry}
\usepackage{mathrsfs}
\usepackage{amsmath, enumerate, graphicx, float}
\usepackage[colorlinks = true,
linkcolor = blue,
urlcolor  = blue,
citecolor = blue,
anchorcolor = blue]{hyperref}
\usepackage{amsfonts}
\usepackage{amssymb, indentfirst}
\usepackage[all]{xy}
\usepackage[toc,page]{appendix}
\usepackage{tikz}
\usetikzlibrary{matrix,arrows}
\tikzset{global scale/.style={
    scale=#1,
    every node/.style={scale=#1}
  }
}
\usetikzlibrary{patterns}

\usepackage{listings}
\usepackage[noadjust]{cite}

\def\NN{\mathbb{N}}
\def\PP{\mathbb{P}}

\def\RR{\mathbb{R}}

\def\ZZ{\mathbb{Z}}

\def\shF{\mathscr{F}}

\def\shO{\mathcal{O}}

\def\shI{\mathcal{I}}

\DeclareMathOperator\conv{conv}

\DeclareMathOperator\Vol{Vol}
\DeclareMathOperator\reg{reg}

\DeclareMathOperator\Cone{Cone}

\DeclareMathOperator\codim{codim}

\DeclareMathOperator\h{H}

\newtheorem{theorem}{Theorem}[section]
\newtheorem{corollary}[theorem]{Corollary}
\newtheorem{proposition}[theorem]{Proposition}
\newtheorem{lemma}[theorem]{Lemma}

\newtheorem{question}[theorem]{Question}
\theoremstyle{remark}
\newtheorem{remark}[theorem]{Remark}
\newtheorem{example}[theorem]{Example}
\theoremstyle{definition}
\newtheorem{definition}[theorem]{Definition}

\begin{document}
\title{On $k$-normality and Regularity of Normal Toric Varieties}
\author{Bach Le Tran}
\email{b.tran@sms.ed.ac.uk}
\address{School of Mathematics\\ University of Edinburgh\\ James Clerk Maxwell Building\\ Peter Guthrie Tait Road\\ Edinburgh EH9 3FD}
\begin{abstract}
We give a bound of $k$ for a very ample lattice polytope to be $k$-normal. Equivalently, we give a new combinatorial bound for the Castelnuovo-Mumford regularity of normal projective toric varieties.
\end{abstract}
\maketitle

\section{Introduction}
Let $L$ be a very ample line bundle on an irreducible projective variety $X$ defining an embedding $X\rightarrow \PP(\h^0(X,L))\cong  \PP^r$. We say that (the embedding of) $X$ is $k$-normal if the restriction map
\[\h^0(\PP^r,\shO_{\PP^r}(k))\rightarrow \h^0(X,\shO_X(k))\]
is surjective. We define the $k$-normality of $X$ to be the smallest positive integer $k_X$ such that $X$ is $k$-normal for all $k\ge k_X$. The $k$-normality of $X$ is closely related to its Castelnuovo-Mumford regularity. We have that $X$ is $(k+1)$-regular if and only if $X$ is $k$-normal and $\shO_X$ is $k$-regular (Proposition \ref{Mumford regularity equivalency}); therefore, $k_X\le \reg(X)-1$.

Now suppose that $X$ is a normal projective toric variety with $L$ a very ample line bundle on $X$. Then $L=\shO_X(D)$ for some torus invariant divisor $D$. Hence, $L$ corresponds to a lattice polytope $P=P_D$. We say that $P$ is $k$-normal if the map
\[\underbrace{P\cap M+\cdots+P\cap M}_{k\text{ times}}\rightarrow kP\cap M\]
is surjective. We also define the $k$-normality of $P$ to be the smallest positive integer $k_P$ such that $P$ is $k$-normal for all $k\ge k_P$. We have $X$ is $k$-normal if and only if $P$ is $k$-normal. Hence, $k_X$ and $k_P$ coincide.

In this paper, we will give a new combinatorial bound of the $k$-normality of very ample lattice polytopes. First of all, define $d_P$ to be the smallest positive integer such that the map
\[P\cap M+kP\cap M\rightarrow (k+1)P\cap M\] is surjective for all $k\ge d_P$. Such a $d_P$ always exists by Lemma \ref{P+(n-1)P=nP}. Since $P$ is very ample, for every vertex $v\in P$, the semigroup $\RR_{\ge 0}(P-v)\cap M$ is generated by $(P-v)\cap M$. Thus, for any lattice point $x\in d_P\cdot P\cap M$ and vertex $d_P\cdot v$ of $d_P\cdot P$, we have
\[x-d_P\cdot v=\sum_{i=1}^m(w_i-v),\]
for some $m<+\infty$, $w_i\in P\cap M$. For such a pair $(x,d_P v)$, we define 
\begin{equation*}\label{definition of sigma(x,v)}
\sigma(x,d_P v)=\min\left\{m\in\NN\hspace*{1mm}\Big\vert\hspace*{1mm}  x-d_Pv=\sum_{i=1}^m(w_i-v) \text{ for some }w_i\in P\cap M\right\}.
\end{equation*}

Let
\[m_P=\max \left\{\sigma(x,d_P v)\hspace*{1mm}\big\vert \hspace*{1mm} x\in (d_P P)\cap M, \text{ } v \text{ a vertex of }P\right\}.\]
We now state the most important corollary of our main result, Theorem \ref{New bound of k-normality}, as follows.

\begin{corollary}\label{Bound of k-normality}
	Suppose that $P$ is a very ample lattice polytope with $n$ vertices. Then
	\[k_P\le (m_P-d_P)\cdot n+1.\]
\end{corollary}

It is then natural to ask for an upper bound of $m_P$. If $P$ is a smooth polytope, we obtained the following result.
\begin{corollary}\label{Smooth corollary in introduction}
Let $P$ be a smooth $d$-dimensional lattice polytope with $n$ vertices. Let $\gamma$ be the smallest integer such that $P\subseteq C_{v, \gamma}:=\conv(v, v+\gamma\cdot(w_{E_1}-v),\cdots, v+\gamma\cdot(w_{E_d}-v))$ for any vertex $v\in P$, where the $(w_{E_i}-v)$'s are the primitive ray generators of the edges of $P$ coming from $v$. Then $P$ is $k$-normal for all 
\[k\ge (\gamma-1)\cdot (d-1)\cdot n+1.\]
\end{corollary} 

Finding the explicit value of $k_P$ is a really hard question in general. Beck et. al. (\cite{Beck2015}) showed that the $k$-normality of the polytope in Example $\ref{Bruns example}$ is $s-1$. With their notations, we have $k_P=\gamma(P)+1$ where $\gamma(P)$ is the largest height that contains gaps in $M_P$. There are also results by Higashitani (\cite{Higashitani2014}), Larso\'{n} and Micha\l ek (\cite{Lason2017}) that give $k_P$ for some classes of lattice polytopes.  Oda (\cite{Oda1988}) asked if $P$ is smooth, is it always the case that $k_P=1$? Despite the simple statement, it is still an open question at the time of writing.  For bounds of $k$-normality, Ogata (\cite[Theorem 2]{Ogata2005}) proved that any projective toric variety of dimension $n\ge 4$ which is a quotient of the projective $n$-space by a finite abelian group embedded by a very ample line bundle in $\PP^r$ is $k$-normal for every $k\ge n-1+[n/2]$. Equivalently, any $n$-dimensional very ample lattice simplex is $k$-normal for $k\ge  n-1+[n/2]$.

The main motivation for the study of $k$-normality is its relation to the Castelnuovo-Mumford regularity, an important invariant in algebraic geometry. First of all, the regularity measures the complexity of the ideal sheaf $\shI_X$ from the perspective of free resolution and gives a bound for the maximal degree of the defining equations of projective varieties. It also gives bound of complexity for algorithms calculating minimal free resolution of ideals generated by finitely many homogeneous polynomials (\cite{Mayr1982, Buchberger1983}). There has been a big focus on finding upper bounds for the Castelnuovo-Mumford regularity of varieties in general. Mumford (\cite{Bayer92}) proved that if $X\subset \PP^r$ is a reduced smooth subscheme purely of dimension $d$ in characteristics $0$, then 
$\reg(X)\le (d+1)(\deg(X)-2))+2$. Kwak (\cite{Kwak1998a}) proved that if $X$ is a smooth variety of dimension $d$ in $\PP^r$ then $\reg(X)\le \deg(X)-\codim(X)+2$ if $d=3$ and $\reg(X)\le \deg(X)-\codim(X)+5$ if $d=4$. Recently Kwak and Park (\cite{Kwak2014}) obtained an upper bound for the regularity of non-degenerate smooth projective varieties; however, it is very hard to find explicit bounds for particular cases. 

For toric varieties, Peeva and Sturmfels proved that for a projective toric variety $X$ of codimension $2$ in $\PP^{d-1}$, not contained in any hyperplane then $\reg(X)\le \deg(X)-1$ (\cite{Peeva1998}, \cite[Theorem 4.2]{Sturmfels1995}). Sturmfels also proved that if $X$ is a projective toric variety in $\PP^{d-1}$ then $\reg(X)\le d\cdot \deg(X)\cdot \codim(X)$ (\cite[Theorem 4.5]{Sturmfels1995}). 

The most well-known question in finding upper bounds for the regularity of projective varieties is the Eisenbud-Goto (\cite{Eisenbud1984}) conjecture which says that if $X$ is irreducible and reduced then
\[\reg(X)\le \deg(X)-\codim(X)+1.\]
Even though the conjecture fails in general (\cite{McCullough2016}), it has motivated many results on regularity. In some particular cases, the Eisenbud-Goto conjecture is proven to be true: smooth surfaces in characteristic zero (\cite{Lazarsfeld1987}), connected reduced curves (\cite{Giaimo2005}), etc. Furthermore, Eisenbud and Goto (\cite{Eisenbud1984}) proved that their conjecture holds when $X$ is arithmetically Cohen-Macaulay. Therefore, it holds for projectively normal toric varieties since they are arithmetically Cohen-Macaulay (\cite{Hochster1972}). For a more detailed list of cases where the Eisenbud-Goto conjecture holds, please refer to \cite{Nitsche2014}. Note that the Eisenbud-Goto conjecture is still open for toric varieties.

Combinatorially, for a normal projective toric variety $X$ embedded in $\PP^r$ via a very ample line bundle with the corresponding lattice polytope $P$, we have $\deg(X)=\Vol(P)$, the normalized volume of $P$, and $\codim(X)=|P\cap M|-\dim P-1$. We define the degree of $P$, denoted by $\deg(P)$, as follows. If $P$ has no interior lattice points, let $\deg(P)$ be the smallest non-negative integer $i$ such that $kP$ contains no interior lattice points for $1\le k\le d-i$. If $P$ has interior lattice point(s) then we define $\deg P=d$. By Proposition \ref{kP=reg(X)+1}, we have $\reg(X)=\max\{k_P,\deg(P)\}+1$. Hence, the Eisenbud-Goto conjecture can be translated as if 
\begin{equation*}
\max\{k_P,\deg(P)\}\le \Vol(P)-|P\cap M|+\dim P+1.
\end{equation*}
Note that by \cite[Proposition 2.2]{Hofscheier2017}, we have \[\deg(P)\le \Vol(P)-|P\cap M|+\dim P+1.\]
Hence, it remains to verify if
\[k_P\le \Vol(P)-|P\cap M|+\dim P+1.\]
In particular, if $k_P\le \deg(P)$ then the Eisenbud-Goto conjecture holds in this case. Our bound in Theorem \ref{New bound of k-normality} proves the conjecture for cases; in particular, the case $s=4$ in Example \ref{Bruns example}.
\subsection*{Acknowledgements}
The author wants to thank Jinhuyng Park for introducing the question and a useful discussion, Arend Bayer, David Cox, and Lukas Katth\"{a}n for some helpful suggestions and conversations, and Milena Hering for all the invaluable advice.
\section{$k$-normality of Very Ample Polytopes}
The main goal of this session is to give a bound for $k$-normality of very ample lattice polytopes. There are a few equivalent definitions of very-ampleness, so to avoid any confusion we introduce the definition used in this paper.

\begin{definition}
A lattice polytope $P\subseteq M_{\RR}$ is very ample if for every vertex $v\in P$, the semigroup $S_{P,v}=\NN(P\cap M-v)$ generated by the set 
\[P\cap M-v=\{u-v\vert u\in P\cap M\}\]
is saturated in $M$; i.e., if $c\cdot u\in S_{P,v}$ for some $c\in \NN^*$ then $u\in S_{P,v}$. Equivalently, $P$ is very ample if the affine monoid $\RR_{\ge 0}(P-v)\cap M$ is generated by the set $P\cap M-v$.
\end{definition}

The starting point for our main result is the following lemma:
\begin{lemma}\label{P+(n-1)P=nP}
Let $P$ be a $d$-dimensional lattice polytope with $n$ vertices $\mathcal{V}=\{v_1,\cdots, v_n\}$. 
\begin{enumerate}[(a)]
	\item For any $k\ge n-1$,
	\[(k+1)P\cap M=\mathcal{V}+kP\cap M.\]
	\item {\cite{Ewald1991, Liu1993, Bruns1997}} For $k\ge d-1$, we have
	\[(k+1)P\cap M=P\cap M+kP\cap M.\]
\end{enumerate}
\end{lemma}
\begin{proof}
	We follow the argument in \cite{Ewald1991} to give a proof for $(a)$. Let $x$ be a lattice point in $(k+1)P\cap M$. Then $x=\sum_{i=1}^n\lambda_iv_i$ for some $\lambda_i\ge 0$, $\sum_{i=1}^n\lambda_i=k+1$. Since $k+1\ge n$, there must be an $i$ such that $\lambda_i\ge 1$. Then
	\begin{align*}
	x=v_i+\left(\sum_{\substack{1\le j\le n \\j\neq i}}\lambda_jv_j+(\lambda_i-1)v_i\right)
	\end{align*}
	with $\left(\sum_{\substack{1\le j\le n \\j\neq i}}\lambda_jv_j+(\lambda_i-1)v_i\right)\in kP\cap M$. The conclusion follows.
\end{proof}

From the above lemma, we obtain two well-defined invariants of  $P$ as follows.

\begin{definition}
Let $P$ be a lattice polytope with the set of vertices $\mathcal{V}=\{v_1,\cdots, v_n\}$. We define $d_P$ to be the smallest positive integer such that the map
\[(k+1)P\cap M=P\cap M+kP\cap M.\]
We also define $\nu_P$ to be the smallest positive integer such that for any $k\ge \nu_P$,
\[(k+1)P\cap M=\mathcal{V}+kP\cap M.\]
\end{definition}

Moreover, it is clear from the definitions that $d_P\le \nu_P\le n-1$. Also, if $P$ is an empty lattice polytope then $\nu_P=d_P$.

%
In general, $k$-normality does not imply $(k+1)$-normality, with some counterexamples given by \cite[Theorem 9]{Handelman1990} and \cite[Theorem 12]{Lason2017}. However, if $P$ is $k$-normal for some $k\ge d_P$ then it is $(k+1)$-normal. We will show this as a part of the following lemma.

\begin{lemma}\label{d_P <= m_P <= k_P}
	Let $P$ be a $d$-dimensional lattice polytope of dimension $d$ with $n$ vertices $\mathcal{V}=\{v_1,\cdots,v_n\}$. 
	\begin{enumerate}[(a)]
		\item \label{2.4.a} For any $k\ge d_P$ and $u\in kP\cap M$, we can write $u$ as
		\[u=x+\sum_{i=1}^{k-d_P}u_i,\]
		where $x\in d_PP\cap M$ and $u_i\in P\cap M$ for all $1\le i\le k-d_P$. 
		\item \label{2.4.b} If $k\ge \nu_P$, for any $u\in kP\cap M$,
		\[u=x+\sum_{i=1}^{\nu_P-d_P}u_i+\sum_{i=1}^n\lambda_iv_i\]
		for some $x\in d_PP\cap M$, $u_i\in P\cap M$, $\lambda_i\in\NN$ such that $\sum_{i=1}^n \lambda_i=k-\nu_P$.
		\item \label{2.4.c} If $P$ is $k$-normal for some positive integer $k\ge d_P$ then $P$ is $(k+1)$-normal.
		\item \label{2.4.d} $d_P\le m_P\le k_P$.
		\item \label{2.4.e} $P$ is normal $\Leftrightarrow d_P=k_P\Leftrightarrow m_P=k_P\Leftrightarrow d_P=m_P$. Therefore, if $P$ is not normal then $k_P\ge m_P\ge d_P+1$.
	\end{enumerate}

\end{lemma}
\begin{proof}
	\begin{enumerate}[(a)]
		\item By the definition of $d_P$, we have a surjective map
		\[d_PP\cap M+\underbrace{P\cap M+\cdots+P\cap M}_{k-d_P}\twoheadrightarrow kP\cap M.\]
		Hence, for any $k\ge d_P$ and $u\in kP\cap M$, we can write $u$ as
		\[u=x+\sum_{i=1}^{k-d_P}u_i,\]
		where $x\in d_PP\cap M$ and $u_i\in P\cap M$ for all $1\le i\le k-d_P$.
		\item Similarly, for $k\ge \nu_P$, we have a surjection
		\[d_PP\cap M+\underbrace{P\cap M+\cdots+P\cap M}_{\nu_P-d_P}+\underbrace{\mathcal{V}+\cdots+\mathcal{V}}_{k-\nu_P}\twoheadrightarrow kP\cap M,\]
		which yields the conclusion.
		\item Suppose that $P$ is $k$-normal; i.e.,
		\[\underbrace{P\cap M+\cdots+P\cap M}_{k}\twoheadrightarrow kP\cap M,\]
		then by the definition of $d_P$ and since $k\ge d_P$, we have
		\[\underbrace{P\cap M+\cdots+P\cap M}_{k}+P\cap M\twoheadrightarrow kP\cap M+P\cap M\twoheadrightarrow (k+1)P\cap M.\]
		In other words, $P$ is $(k+1)$-normal.
		\item It follows from the definitions that $k_P\ge d_P$. Now let $x$ be any lattice point in $d_PP\cap M$ and $v$ a vertex of $P$. Then since $x+(k_P-d_P)v\in kP\cap M$, there exists $w_i\in P\cap M$, $i=1,\cdots, k$ such that $x+(k_P-d_P)v=\sum_{i=1}^{k_P}w_i$. In other words,
		\[x-d_Pv=\sum_{i=1}^{k_P}(w_i-v).\]
		Therefore, $\sigma(x,d_Pv)\le k_P$. Hence, $m_P\le k_P$.
		
		For any vertex $v$ of $P$, let $w\in P\cap M$ be a point with maximal distance from $v$. We have
		\[d_Pw-d_Pv=\sum_{i=1}^{\sigma(d_Pw,d_Pv)}(w_i-v)\]
		for some $w_i\in P\cap M$.  Then
		\begin{align*}
		\left\Vert d_Pw-d_Pv\right\Vert =\left\Vert\sum_{i=1}^{\sigma(d_Pw,d_Pv)}(w_i-v)\right\Vert	\le &\sum_{i=1}^{\sigma(d_Pw,d_Pv)}\left\Vert(w_i-v)\right\Vert\\
		\le &\sigma(d_Pw,d_Pv)\cdot \left\Vert w-v\right\Vert.
		\end{align*}
		It follows that $\sigma(d_Pw,d_Pv)\ge d_P$. Therefore, $m_P\ge d_P$.
		\item $\bullet$ $P$ is normal $\Leftrightarrow m_P=k_P$: if $P$ is normal then $k_P=d_P=1$. Conversely, suppose that $k=d_P=k_P\ge 2$. Then by the definitions of $k_P$ and $d_P$, we have a surjection
		\[\underbrace{P\cap M+\cdots+P\cap M}_{k}\twoheadrightarrow kP\cap M,\]
		while the map
		\[(k-1)P\cap M+P\cap M\rightarrow kP\cap M\]
		is not surjective. This is a contradiction because 
		\[\underbrace{P\cap M+\cdots+P\cap M}_{k}\subseteq (k-1)P\cap M+P\cap M.\]
		Hence, $P$ must be normal in this case.
		
		\noindent $\bullet$ $P$ is normal $\Leftrightarrow d_P=k_P$: If $P$ is normal then $k_P=d_P=1$. Conversely, suppose that $d_P=k_P$. Then by (\ref{2.4.d}), we have $m_P=k_P$. Therefore, $P$ is normal by the first equivalence.
		
		\noindent $\bullet$ $P$ is normal $\Leftrightarrow m_P=d_P$: if $P$ is normal then it is clear that $m_P=d_P=1$. Suppose conversely that $m_P=d_P$. Then for any vertex $v\in P\cap M$ and any $x\in d_PP\cap M$, we have $x-d_Pv=\sum_{i=1}^{d_P}(w_i-v)$ which implies $x=\sum_{i=1}^{d_P}w_i$ for some $w_i\in P\cap M$. Then $P$ is $d_P$-normal, so $d_P=k_P$ by the first part of this lemma. Thus $P$ is normal by the second equivalence.
	\end{enumerate}
\end{proof}

\begin{remark}The above lemma plays a crucial role in our main result in this section. Notice that $m_P=k_P$ does not implies $P$ is normal. A counterexample is given by the case $s=4$ in Example \ref{Bruns example}.
\end{remark}
\begin{remark}
Let $P$ be a $d$-dimensional lattice polytope. Let $\mathcal{LD}_P(n)$ be the property that for any $k\ge n$ and $u\in kP\cap M$, we can write $u$ as
\[u=x+\sum_{i=1}^{k-n}u_i,\]
where $x\in nP\cap M$ and $u_i\in P\cap M$ for all $1\le i\le k-n$. Then \[d_P=\min\{n\in \NN\hspace{1mm}\vert\hspace{1mm} \mathcal{LD}_P(n)\text{ holds}\}.\]
	
Indeed, suppose that $N=\min\{n\in \NN\hspace{1mm}\vert\hspace{1mm} \mathcal{LD}_P(n)\text{ holds}\}$. Then we have a surjection
\[kP\cap M+P\cap M\twoheadrightarrow (k+1)P\cap M\]
for all $k\ge N$. Therefore, $N\ge d_P$. On the other hand, by Lemma \ref{d_P <= m_P <= k_P}, $\mathcal{LD}_P(d_P)$ holds so $N\le d_P$ because of the minimality of $N$. Hence, $N=d_P$. The conclusion follows.
\end{remark}

\begin{theorem}\label{New bound of k-normality}
Suppose that $P$ is a very ample lattice polytope with $n$ vertices. Then
\[k_P\le (m_P-d_P)\cdot n+1.\]
The equality occurs if and only if $P$ is normal. Furthermore, if $P$ is not normal then
\[k_P\le (m_P-d_P-1)\cdot n+\nu_P+1.\]
\end{theorem}

\begin{proof}
If $P$ is normal then $m_P=d_P=k_P=1$ by Lemma \ref{d_P <= m_P <= k_P} (\ref{2.4.e}). Assume that $P$ is not normal, since $\nu_P\le n-1$, it is enough to show that
\[k_P\le (m_P-d_P-1)\cdot n+\nu_P+1.\]
By Lemma \ref{d_P <= m_P <= k_P} (\ref{2.4.e}), we have $m_P\ge d_P+1$. Let $\mathcal{V}=\{v_1,\cdots, v_n\}$ be the set of vertices of $P$. Let $k\ge(m_P-d_P-1)\cdot n+\nu_P+1$ and $p\in kP\cap M$. Notice that $(m_P-d_P-1)\cdot n+\nu_P+1\ge \nu_P+1$, by Lemma \ref{d_P <= m_P <= k_P} (\ref{2.4.b}) the lattice point $p$ of $kP$ can be written as
\begin{equation}\label{p=x+sum u_i+sum lamda_iv_i}
p=x+\sum_{i=1}^{\nu_P-d_P}u_i+\sum_{i=1}^n\lambda_iv_i
\end{equation}
for some $x\in d_PP\cap M$, $u_i\in P\cap M$, and $\lambda_i\in\ZZ_{\ge 0}$ such that $\sum_{i=1}^n\lambda_i=k-\nu_P$. Now $k\ge (m_P-d_P-1)\cdot n+\nu_P+1$ implies that $k-\nu_P\ge (m_P-d_P-1)n+1$. Thus, by the pigeonhole principle, there must be an $i$ such that $\lambda_i\ge m_P-d_P$. Without loss of generality, assume that $\lambda_1\ge m_P-d_P$. Since $P$ is very ample, we can write
\begin{equation}\label{x-d_Pv_1=sum a_i(w_i-v_1)}
x-d_Pv_1=\sum_{i\in I}a_i(w_i-v_1)
\end{equation}
for some $a_i\in\NN$ and $w_i\in P\cap M$ with $\sum_{i\in I}a_i\le m_P$. Substituting Equation (\ref{x-d_Pv_1=sum a_i(w_i-v_1)}) into Equation (\ref{p=x+sum u_i+sum lamda_iv_i}) yields
\begin{align*}
p&=d_Pv_1+\sum_{i\in I}a_i(w_i-v_1)+\sum_{i=1}^{\nu_P-d_P}u_i+\sum_{i=1}^n\lambda_iv_i\\
&=\left(d_P+\lambda_1-\sum_{i\in I}a_i\right)v_1+\sum_{i\in I}a_iw_i+\sum_{i=1}^{\nu_P-d_P}u_i+\sum_{i=2}^n\lambda_iv_i.
\end{align*}
The sum of the coefficients in the last line is $k$ and each of them is non-negative since $d_P+\lambda_1\ge m_P\ge \sum_{i\in I}a_i$. Hence, $p$ can be written as a sum of $k$ lattice points in $P$; i.e., $P$ is $k$-normal. Therefore, $k_P\le (m_P-d_P-1)\cdot n+\nu_P+1$.

Now suppose that $k_P=(m_P-d_P)\cdot n+1$ but $P$ is not normal. Then
\[(m_P-d_P)\cdot n+1\le (m_P-d_P-1)\cdot n+\nu_P+1\]
which implies $\nu_P\ge n$, a contradiction. Hence, $P$ must be normal.
\end{proof}

\begin{remark}\text{ }
If $P$ is normal, then $m_P=d_P=1$ and $(m_P-d_P)\cdot n+1=1$. Our bound is sharp for this case. Another case where our bound is sharp is given in Example \ref{Bruns example}.
\end{remark}

The following example gives a comparison between known results on $k$-normality of polytopes with our result in Theorem \ref{New bound of k-normality} for the case of unit hypercubes.
\begin{example}\label{hyper cube example}
	Consider the unit $d$-dimensional hypercube $P$ with $X$ the toric variety obtained from $P$. Then we know that $d_P=1$ and it follows that $m_P=1$. Our bound in Theorem \ref{New bound of k-normality} implies that $k_P=1$. This bound is sharp.	
	We have the following table of known bounds of $k_P$.
	
	\begin{center}
		\begin{tabular}{|c|c|c|c|c|}
			\hline
			$k_P$&Theorem \ref{New bound of k-normality} & Mumford (\cite{Bayer92}) & Sturmfels (\cite{Sturmfels1995}) & Eisenbud-Goto (\cite{Eisenbud1984}) \\
			\hline
			$1$&$1$ & $(d+1)(d!-2)+1$ & $2^d\cdot (d!)(2^d-d-1)-1$ & $d!-2^d+d+1$ \\
			\hline 
		\end{tabular}
	\end{center}
\end{example}

The only occasion where we need very-ampleness in the proof of Theorem \ref{New bound of k-normality} is to define $m_P$. Thus, if we assume $m_P$ is defined for an arbitrary lattice polytope $P$, it follows that $P$ is $k$-normal for $k$ big enough. We obtain the following criterion for a lattice polytope to be very ample.
\begin{proposition}
	Let $P$ be a lattice polytope. Then $P$ is very ample if and only if there exists $r\ge d_p$ such that for any $x\in rP\cap M$ and $v$ a vertex of $P$ we have
	\[x-rv=\sum_{i=1}^n(w_i-v)\]
	for some $n<\infty$ and $w_i\in P\cap M$.
\end{proposition}
\begin{proof}
The ``only if" part follows directly from the definition of very ample polytopes. We now prove the ``if" part. For an $r\ge d_P$, define
\[m=\max\left\{n\in\NN\bigg\vert x-rv=\sum_{i=1}^n(w_i-v), w_i\in P\cap M\right\}.\]
It follows from the proofs of Lemma \ref{d_P <= m_P <= k_P} (\ref{2.4.d}) and Theorem \ref{New bound of k-normality} that $m\ge d_P$ and
\[k_P\le (m-d_P)\cdot n+1.\]
Then $P$ is $k$-normal for $k\gg 0$, which implies that $P$ is very ample. The conclusion follows.
\end{proof}

\section{Bounds of $m_P$ and Aplications on Smooth Polytopes}
\label{section on bounds of m}

In this section, we will give some bounds for $m_P$ depending on the combinatorial data of smooth lattice polytope $P$. We begin by introducing some standard facts about polytopes.

\begin{definition}
Let $P$ be a lattice polytope of dimension $d$. The normalized volume of $P$, denoted by $\Vol(P)$, is defined to be 
\[\Vol(P)=d!\cdot\{\text{Euclidean volume of } P\}.\]
\end{definition}

The following classical lemma gives a straightforward way to calculate the normalized volume of any lattice polytope given the coordinates of its vertices.

\begin{lemma}\label{Simplex volume}
Let $P$ be a $d$-simplex with vertices $\{v_0,\cdots,v_d\}$. Then the normalized volume of $P$ is given by
\[\Vol(P)=|\det(v_1-v_0,\cdots,v_d-v_0)|.\]
\end{lemma}
Using normalized volume, we obtain our first bound of $m_P$:
\begin{proposition}\label{Bound of m2 in term of volume}
Let $P$ be a smooth $d$-dimensional lattice polytope. Then for every $x\in (d_P\cdot P)\cap M$ and $v$ a vertex of $P$, we have
\[m_P\le d\cdot d_P^d\cdot\Vol(P).\]
\end{proposition}

\begin{proof}
Let $\mathcal{V}=\{v_1,\cdots,v_n\}$ be the vertices of $P$. For any $x\in d_P\cdot P\cap M$, and $v\in\mathcal{V}$, since $P$ is smooth, we have
\[x-d_P\cdot v=\sum_{i=1}^d a_{E_i}(w_{E_i}-v),\]
for some $a_i\in\ZZ_{\ge 0}$ with $w_{E_i}-v$ is the primitive ray generator of $\Cone(v_{E_i}-v)$, where $v_{E_i}$ is a neighbor of $v$. By Cramer's rule
\[a_{E_i}=\frac{\Delta_{E_i}}{\Delta},\]
where 
\[\Delta=\det(w_{E_1}-v,\cdots,w_{E_d}-v)=1,\text{ and }\]
\[\Delta_{E_i}=\det(w_{E_1}-v,\cdots,w_{E_{i-1}}-v,x-d_P\cdot v,w_{E_{i+1}}-v\cdots,w_{E_d}-v).\] 
By Lemma \ref{Simplex volume}, $\Delta_{E_i}$ is the normalized volume of the simplex 
\[\conv(d_P v,w_{E_1}+(d_P
-1)v,\cdots,w_{E_{i-1}}+(d_P
-1)v,x,w_{E_{i+1}}+(d_P
-1)v,\cdots,w_{E_d}+(d_P
-1)v),\]
which lies inside the polytope $d_P\cdot P$. Thus, $\Delta_{E_i}\le \Vol(d_P\cdot P)=d_P^d\Vol(P)$. Therefore,
\[\sigma(x,d_P\cdot v)=\sum_{i=1}^da_{E_i}\le d\cdot d_P^d\Vol(P).\]
\end{proof}

Let $P$ be a $d$-dimensional smooth lattice polytope. Then for each vertex $v$ of $P$, there exist $d$ neighbor vertices to $v$, say $v_{E_i},\cdots, v_{E_d}$, with $E_i$ is the edge of joining $v$ with $v_{E_i}$. Let $w_{E_i}-v$ be the primitive ray generator of $\Cone(v_{E_i}-v)$. We define the corner of $P$ at $v$, a vertex of $P$, to be
\[C_{v}:=\conv\left(v,w_{E_i},\cdots,w_{E_d}\right),\] and the $\gamma$-scaling of $C_v$ to be
\[C_{v, \gamma}=\conv(v, v+\gamma\cdot(w_{E_1}-v),\cdots, v+\gamma\cdot(w_{E_d}-v)).\]
Then for some $\gamma$ big enough, $C_{v,\gamma}$ contains the whole polytope $P$.

\begin{example}\label{gamma-scaling}
Consider the polytope $P$ given by
\[P=\conv\left(
\begin{matrix}
0 & 2 & 2 & 1 & 0\\
0 & 0 & 1 & 2 &2
\end{matrix}
\right).\]
Let $O$ be the origin and $v=(1,1)^T$. Then 
\[C_{O}=\conv\left(
\begin{matrix}
0 & 1 & 0\\
0 & 0 & 1
\end{matrix}
\right) \text{ and } C_{v}=\conv\left(
\begin{matrix}
2 & 2 & 1\\
1 & 0 & 1
\end{matrix}
\right).\]
We have $C_{v,5}\supseteq P$ and $C_{O,3}\supseteq P$.
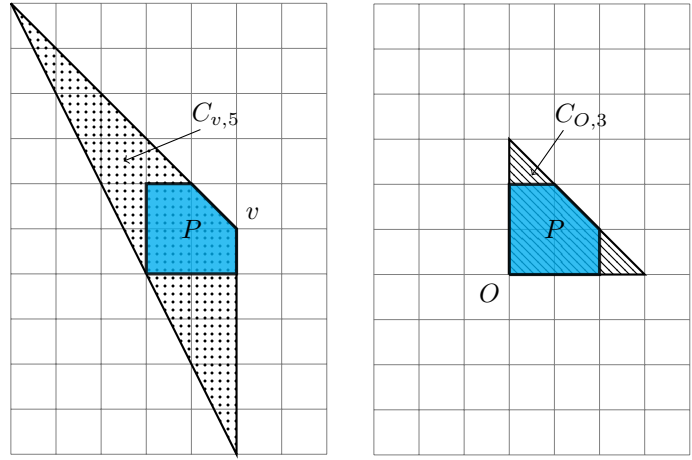
	\begin{center}
		\begin{figure}[h]
			\begin{tikzpicture}[scale=0.6]
				\draw [help lines] (0,0) grid (7,10);
				
				\draw [pattern=dots, thick] 
					(0,10)--(5,5)--(5,0)--(0,10);
				\draw [fill=cyan, very thick, opacity=0.8] 		
					(3,4)--(3,6)--(4,6)--(5,5)--(5,4)--(3,4);
				\node [above right] at (5,5) {$v$};
				\node at (4,5){$P$};
				\node [right] at (3.8,7.5){$C_{v,5}$};
				\draw[->] (4.2,7.2) -- (2.5,6.5);
			\end{tikzpicture}\hspace*{5mm}
			\begin{tikzpicture}[scale=0.6]
				\draw [help lines] (0,0) grid (7,10);
				
				\draw [ pattern=north west lines, thick]
					(3,4)--(3,7)--(6,4)--(3,4);
				\draw [fill=cyan, very thick, opacity=0.8] 		
					(3,4)--(3,6)--(4,6)--(5,5)--(5,4)--(3,4);
				\node [below left] at (3,4) {$O$};
				\node at (4,5){$P$};
				\node [right] at (3.8,7.5){$C_{O,3}$};
				\draw[->] (4.2,7.2) -- (3.5,6.2);
			\end{tikzpicture}	
		\end{figure}
	\end{center}
\end{example}

\begin{proposition}\label{bounds of m1, m2 for smooth polytopes}
Let $P$ be a smooth $d$-dimensional lattice polytope with $\mathcal{V}=\{v_1,\cdots, v_n\}$ the set of its vertices. Let $\gamma$ be the miminum interger such that $P\subseteq C_{v_i,\gamma}$ for every $1\le i\le n$. Then for any $u\in d_P\cdot P\cap M$, $v\in \mathcal{V}$,
\[\sigma(u,d_P\cdot v)\le d_P\cdot \gamma,\]
which implies $m_P \le d_P\cdot \gamma$.
\end{proposition}
\begin{proof}
For any lattice point $u\in mP\cap M$ (with $m\in\ZZ_{\ge 1}$) and vertex $v$ of $mP$, $u$ lies inside the $d$-simplex formed by scaling the corner at $v$ by $m\gamma$. Precisely,
\[u\in  C_{v, m\gamma}=\conv(v, v+m\gamma\cdot(w_{E_1}-v),\cdots, v+m\gamma\cdot(w_{E_d}-v)),\]
where $w_{E_i}-v$ is the primitive ray generator of $v_{E_i}-v$. Equivalently, there are $\lambda_i\ge 0$ with $\sum_{i=0}^d\lambda_i=1$ such that
\begin{equation}
u=\lambda_0v+\sum_{i=1}^d\lambda_i\cdot(v+m\gamma\cdot(w_{E_i}-v))=v+\sum_{i=1}^d \lambda_i\cdot m\cdot \gamma\cdot(w_{E_i}-v).
\end{equation}
Hence,
\begin{equation}\label{first presentation of u-v}
u-v=\sum_{i=1}^d \lambda_i\cdot m\cdot \gamma\cdot (w_{E_i}-v).
\end{equation}
Since $mP$ is smooth at $v$, $C_{v,m\gamma}$ is also smooth at $v$, and we can express $u-v$ uniquely in the form
\begin{equation}\label{second presentation of u-v}
u-v=\sum_{i=1}^d a_i(w_{E_i}-v),
\end{equation}
where $a_i\in\NN$ for $1\le i\le d$. Comparing the coefficients in the equations (\ref{first presentation of u-v}) and (\ref{second presentation of u-v}) yields 
\begin{equation}\label{equality of coefficients of u-v}
a_i=\lambda_i\cdot m\cdot \gamma.\tag{$\dagger$}
\end{equation}
Applying (\ref{equality of coefficients of u-v}) for $m=d_P$ and $u\in d_PP\cap M$, we have
\[\sum_{i=1}^da_i=\sum_{i=1}^d d_P\cdot\lambda_i\cdot \gamma =d_P\cdot \gamma\cdot \sum_{i=1}^d\lambda_i\le d_P\cdot \gamma.\]
In other words,
\[\sigma(u,d_P\cdot v)\le d_P\cdot\gamma.\]
In particular, since $m_P$ is the maximum of the $\sigma(u,d_P\cdot v)$, we have
\[m_P\le d_P\cdot\gamma.\]
\end{proof}

As a corollary, we obtain a bound for smooth lattice polytopes as follows.
\begin{corollary}\label{Smooth corollary 1}
Let $P$ be a smooth $d$-dimensional lattice polytope with $n$ vertices, $\gamma$ is the minimum integer such that $P\subseteq C_{v,\gamma}$ for every vertex $v$ of $P$. Then $P$ is $k$-normal for all 
\[k\ge \min\left\{
\begin{array}{l}
d_P(\gamma-1)\cdot n+1,\\
\left( d\cdot d_P^d\cdot \Vol(P)-d_P\right)\cdot n+1
\end{array}
\right\}.\]
\end{corollary}
\begin{proof}
This follows directly from Theorem \ref{New bound of k-normality}, Proposition \ref{Bound of m2 in term of volume}, and Proposition \ref{bounds of m1, m2 for smooth polytopes}.
\end{proof}
Corollary \ref{Smooth corollary in introduction} follows since $d_P\le d-1$.

\begin{remark}
As a final remark to this session, suppose that $P$ is a $d$-dimensional smooth lattice polytope. Then for any lattice point $u\in P\cap M$ and any vertex $v\in\mathcal{V}$, we have
\begin{equation}
u-v=\sum_{i=1}^d a_i(w_{E_i}-v),
\end{equation}
where $a_i\in \ZZ_{\ge 0}$ and $w_{E_i}-v$ is the primitive generator of $\Cone(v_{E_i}-v)$. Take $m'$ to be the maximal of all such $a_i$; i.e.,
\[m'=\max_{u\in P\cap M, v\in\mathcal{V}}\left\{a_i\hspace{1mm}\Big\vert \hspace{1mm} u-v=\sum_{i=1}^d a_i(w_{E_i}-v)\right\}.\]
Then $m'$ is well-defined because $P$ is a smooth polytope and $|P\cap M|<\infty$. We have $P\cap M\subseteq C_{v,m'}$ for every $v\in\mathcal{V}$. In other words, $\gamma\le m'$.
\end{remark}

\section{The Castelnuovo-Mumford Regularity of Normal Toric Varieties}
In this section, we will give a survey on combinatorial interpretations of the Eisenbud-Goto conjecture. First of all, let us recall the definition of Castelnuovo-Mumford regularity.
\begin{definition}
	Let $X\subset \PP^r$ be a projective variety and $\shF$ a coherent sheaf over $X$. We say that $\shF$ is $k$-regular if
	\[\h^i(X,\shF(k-i))=0\]
	for all $i>0$. 
The regularity of $\shF$, denoted by $\reg(\shF)$, is the minimum number $k$ such that $\shF$ is $k$-regular. We also say that $X$ is $k$-regular if the ideal sheaf $\shI_X$ of $X$ is $k$-regular and use $\reg(X)$ to denote the regularity of $X$ (or of $\shI_X$).
\end{definition}


Regularity and $k$-normality are closely related by the well-known fact as noted, for example, in \cite{Kwak1998} as follows.


\begin{proposition}\label{Mumford regularity equivalency}
	Let $X\subseteq \PP^r$ be an irreducible projective variety. Then for $k\in\ZZ_{\ge 1}$, $X$ is $(k+1)$-regular if and only if $X$ is $k$-normal and $\shO_X$ is $k$-regular.
\end{proposition}
\begin{proof}
It is clear that $X$ is $k$-normal if and only if $\h^1(\shI_X(k))=0$. We have an exact sequence
	\[0\rightarrow \shI_X\rightarrow \shO_X\rightarrow \shO_{\PP^r}\rightarrow 0.\]
Suppose that $X$ is $(k+1)$-regular; i.e., $\h^i(\shI_X(k+1-i))=0$ for all $i\ge 1$. Taking the long exact sequence of the cohomology, we see that for $i=1$, $\h^1(\shI_X(k))=0$; i.e., $X$ is $k$-normal; and that $\h^{i-1}(\shO_X(k+1-i))=0$ since $\h^i(\shO_{\PP^r}(k+1-i))=0$  for $i\ge 2$. In other words, $\shO_X$ is $k$-regular.

Conversely, suppose that $X$ is $k$-normal and $\shO_X$ is $k$-regular. From the long exact sequence of the cohomology, we have for all $i\ge 2$,  $\h^i( \shI_X(k+1-i))=0$ since $\h^{i}(\shO_{\PP^r}(k+1-i))=0$ and $\h^{i-1}(\shO_X(k+1-i))=0$ by the hypothesis that $\shO_X$ is $k$-regular. The case $\h^1(\shI_X(k))=0$ follows from the assumption that $X$ is $k$-normal. Hence, $X$ is $(k+1)$-regular. The conclusion follows.
\end{proof}

As a corollary, we obtain an equation of $\reg(X)$ in terms of $\reg(\shO_X)$ and $k_X$ for any irreducible projective variety $X$.

\begin{proposition}\label{reg(X)=max(k_X,reg(O_X))+1}
Let $X\subseteq\PP^r$ be an irreducible projective variety. Then
\[\reg(X)=\max\{\reg(\shO_X),k_X\}+1.\]
\end{proposition}
\begin{proof}
By a result attributed to Castelnuovo by Mumford (\cite[Lecture 14]{Mumford1966}), if $X$ is $k$-regular then $X$ is $(k+1)$-regular. Therefore, for any $k\ge \max\{\reg(\shO_X), k_X\}$ we have $X$ is $k$-normal and $\shO_X$ is $k$-regular. By Proposition \ref{Mumford regularity equivalency}, $X$ is $(k+1)$-regular. Thus,
\[\reg(X)\le \max\{\reg(\shO_X),k_X\}+1.\]
Now suppose that $k\le \max\{\reg(\shO_X),k_X\}-1$. Then either $X$ is not $k$-normal or $\shO_X$ is not $k$-regular. Hence, $\shI_X$ is not $(k+1)$-regular by Proposition \ref{Mumford regularity equivalency}. Therefore,
\[\reg(X)\ge \max\{\reg(\shO_X),k_X\}+1.\]
The conclusion follows.
\end{proof}

Now we give a combinatorial interpretation of $\reg(\shO_X)$ for $X$ a normal projective toric variety.
\begin{proposition}\label{O_X is k-regular iff k>=deg(P)}
Let $X\subset \PP^r$ be a normal projective toric variety with $P$ the corresponding lattice polytope of the embedding. Then $\shO_X$ is $k$-regular if and only if $k\ge \deg(P)$. In other words, $\reg(\shO_X)=\deg(P)$.
\end{proposition}
\begin{proof}
The ``if" direction essentially follows from \cite[Theorem IV.5]{Hering2006}. Conversely, suppose that $k\le \deg(P)-1$. Then we have
\begin{align*}
\dim \h^d(X,\shO_X(k-d)=|(d-k)P\cap M|\neq 0
\end{align*}
since $d-k\ge d-\deg(P)+1$. Hence, $\shO_X$ is not $k$-regular. The conclusion follows.
\end{proof}
%
Combining Propositions \ref{reg(X)=max(k_X,reg(O_X))+1} and \ref{O_X is k-regular iff k>=deg(P)}, we obtain a combinatorial relation between $\reg(X)$, $k_P$, and $\deg(P)$, the degree of $P$.

\begin{proposition}\label{kP=reg(X)+1}
	Let $X\subset \PP^r$ be a $d$-dimensional normal projective toric variety $X$ and $P$ the corresponding lattice polytope of the embedding of $X$. Then
	\[\reg(X)=\max\{k_P,\deg(P)\}+1.\] 
\end{proposition}
\begin{proof}
This follows directly from  Propositions \ref{reg(X)=max(k_X,reg(O_X))+1} and \ref{O_X is k-regular iff k>=deg(P)}.
\end{proof}

Notice that $\deg(P)\le d$. Thus, using the upper bound of $k_P$ we obtained in Theorem \ref{Bound of k-normality} for Proposition \ref{kP=reg(X)+1}, we obtain an upper bound for $\reg(X)$: 

\begin{corollary}\label{corollary on reg (X)}
	Let $X\subset\PP^r$ be a normal projective toric variety with $P$ the corresponding lattice polytope of the embedding. Suppose that $P$ has $n$ vertices. Then
	\[\reg(X)\le (m_P-d_P)n+1,\]
where $m_P$ and  $d_P$ are defined as in Theorem \ref{New bound of k-normality}.
\end{corollary}
\begin{proof}
	This follows directly from Propositions \ref{Bound of k-normality}, \ref{kP=reg(X)+1}, and the fact that $\deg(P)\le \dim(P)$.
\end{proof}

By \cite[Proposition 2.2]{Hofscheier2017}, we have that for $P$ a very ample lattice polytope,
\[\deg(P)\le \Vol(P)-|P\cap M|+d+1.\]
Combining this with Proposition \ref{kP=reg(X)+1}, we have the following corollary.
\begin{corollary}\label{EG for deg P>kP}
Let $X\subset \PP^r$ be a normal projective toric variety with $P$ the corresponding lattice polytope of the embedding. Suppose that $k_P\le \dim(X)$; i.e., $P$ is $\dim(X)$-normal, then 
	\[\reg(X)\le \dim (X)+1.\]
	Furthermore, if $k_P\le \deg(P)$, then
	\[\reg(X)\le \min\{\dim(X)+1,\deg(X)-\codim(X)+1\}.\]
\end{corollary}
\begin{proof}
Let $d=\dim(X)=\dim(P)$. If $k_P\le d$, then since $\deg(P)\le d$ by definition, $\max\{k_P,\deg(P)\}\le d$. Hence, by Proposition \ref{kP=reg(X)+1},
\[\reg(X)\le \dim (X)+1.\]
Now suppose that $\deg(P)\ge k_P$. Then $k_P\le d$, so we only need to show that $\reg(X)\le\deg(X)-\codim(X)+1\}$ by the argument above. By Proposition \ref{kP=reg(X)+1}, $\reg(X)=\deg(P)+1$. Now we since $\deg(P)\le \Vol(P)-|P\cap M|+d+1$ by \cite[Proposition 2.2]{Hofscheier2017}, it follows that 
\[\reg(X)\le\deg(X)-\codim(X)+1\}.\]
The conclusion follows.
\end{proof}


\begin{remark}
If $k_P=1$, Proposition \ref{EG for deg P>kP} implies the well-known fact that the Eisenbud-Goto conjecture holds for projectively normal toric varieties (\cite{Eisenbud1984, Hochster1972}).
\end{remark}

\begin{remark}\label{Combinatorial EG}
	By Propositions \ref{kP=reg(X)+1} and \ref{EG for deg P>kP}, we can now restate the Eisenbud-Goto conjecture combinatorially as follows: if $P$ is a non-normal very ample $d$-dimensional lattice polytope, then
	\[k_P\le \Vol(P)-|P\cap M|+d+1.\]
Suppose that $P$ has $n$ vertices, then $n\le |P\cap M|$. Thus by Corollary \ref{corollary on reg (X)}, if 
\[(m_P-d_P+1)|P\cap M|+1\le \Vol(P)+d\]
then the Eisenbud-Goto conjecture holds for the toric variety $X$ associated to $P$. Unfortunately, this is not always the case, as we will show in Example \ref{Bruns example}. Furthermore, suppose that $P$ is smooth and $\gamma=\min\{n\in \ZZ_{\ge 1}\mid C_{v,n}\supseteq P\text{ for all }v\in\mathcal{V}\}$, then Eisenbud-Goto conjecture holds for $X$ if 
\[(d_P\cdot (\gamma-1)+1)\cdot|P\cap M|\le \Vol(P)+d.\]
If the Oda's conjecture holds, then the Eisenbud-Goto conjecture would be true for all smooth polytopes as well because of Proposition \ref{EG for deg P>kP}.
\end{remark}

\section{Examples of The Regularity of Some Non-Normal Very Ample Polytopes}
In this section, we will show that the Eisenbud-Goto conjecture holds for some known examples of non-normal very ample polytopes. We first consider the following example by Gubeladze and Bruns.

\begin{example}[\cite{Gubeladze2009}]\label{Bruns example}
Consider the polytope $P$ which is the convex hull of the vertices given by the columns of the following matrix
\[M=\left(
\begin{array}{cccccccccc}
0&1&0&0&1&0&1&1\\
0&0&1&0&0&1&1&1\\
0&0&0&1&1&1&s&s+1
\end{array}
\right)
\]
with $s\ge 4$. We can verify directly that $P$ is not $(s-2)$-normal for $s\ge 4$. Indeed, let $v=(1,1,s-1)^T$. Then $v\in (s-2)P\cap M$ since
\begin{align*}
\left(
\begin{array}{c}
1\\
1\\
s-1
\end{array}
\right)&= \frac{s(s-2)-s-1}{s(s-2)}\left(
\begin{array}{c}
0\\
0\\
0
\end{array}
\right)+\frac{1}{s(s-2)}\left(
\begin{array}{c}
s-2\\
0\\
0
\end{array}
\right)\\
&+\frac{1}{s(s-2)}\left(
\begin{array}{c}
0\\
s-2\\
0
\end{array}
\right)+\frac{s-1}{s(s-2)}\left(
\begin{array}{c}
s-2\\
s-2\\
s(s-2)
\end{array}
\right),
\end{align*}
with
\[1=\frac{s(s-2)-s-1}{s(s-2)}+\frac{1}{s(s-2)}+\frac{1}{s(s-2)}+\frac{s-1}{s(s-2)}.\]
But $v\notin\underbrace{P\cap M+\cdots+P\cap M}_{s-2}$.


Now we have $\Vol(P)=s+6$, $|P\cap \ZZ^3|=8$, $\dim (P)=3$, so let $X$ be the toric variety associated to $P$, we have
\[\deg(X)=s+6\text{ and } \codim(X)=|P\cap\ZZ^3|-(\dim(P)+1)=4.\]
The Eisenbud-Goto conjecture says that $\reg(X)\le \deg(X)-\codim(X)+1=s+3$.
By \cite[Theorem 3.3]{Beck2015}, $k_P=s-1\ge 3\ge\deg(P)$. Hence, by Proposition \ref{EG for deg P>kP}, we have $\reg(X)=s$ and the Eisenbud-Goto conjecture holds for this example $s$. 

To compare this to the bound of Theorem \ref{New bound of k-normality}, we have $d_P=\nu_P=2$, $d_P+1\le m_P\le s-1$ by Lemma \ref{k_P>= m_P>= d_P}, and $k_P\le 8(s-4)+3=8s-29$. We have the following table of the known bounds of the Castelnuovo-Mumford regularity of $X$:

\begin{center}
\begin{tabular}{|c|c|c|c|}
\hline
$\reg(X)$&Theorem \ref{New bound of k-normality} & Sturmfels (\cite{Sturmfels1995}) & Eisenbud-Goto (\cite{Eisenbud1984})\\
\hline
$s$&$8s-28$& $24(s+6)$ & $s+3$\\
\hline
\end{tabular}
\end{center}

For $s=4$, we have $8s-28=4$, so the bound in Theorem \ref{New bound of k-normality} is sharp and the Eisenbud-Goto conjecture holds for this case. For $s\ge 5$, since $8s-28>s+3$, our bound does not imply the Eisenbud-Goto conjecture.
\end{example}

This example is interesting in many ways. First of all, it gives an example of non-normal very ample polytopes. In addition, because of $P$ is not $(s-2)$-normal and Proposition \ref{kP=reg(X)+1}, one cannot bound the $k$-normality and Castelnuovo-Mumford regularity of $X$ by any polynomial of $\dim X$. Furthermore, the polytope $P$ gives an example of very ample polytopes that cannot be covered by very ample simplices. To show this, we need the following lemma.
\begin{lemma}\label{inequalities on k_P}
Let $P_1,\cdots, P_n$ be very ample lattice polytopes such that $P=\bigcup_{i=1}^nP_i$ is a convex polytope. Then $P$ is very ample and
\[k_{P}\le \max\{k_{P_i}\vert i=1,\cdots,n\}.\]
\end{lemma}
\begin{proof}
Suppose that $x\in kP$ for some $k\ge \max\{k_{P_i}\vert i=1,\cdots,n\}$. Then $x\in kP_i$ for some $i=1,\cdots, n$. Hence, $x$ can be expressed as a sum of $k$ lattice points in $P_i\subseteq P$. Therefore, $P$ is $k$-normal for all $k\ge \max\{k_{P_i}\vert i=1,\cdots,n\}$; i.e., $k_{P}\le \max\{k_{P_i}\vert i=1,\cdots,n\}$ and $P$ is very ample since we know that a polytope is very ample if and only if it is $k$-normal for some $k$ big enough. 
\end{proof}

\begin{proposition}\label{3d non-normal cannot be covered by very ample simplices}	
Any $3$-dimensional very ample non-normal lattice polytope $P$ cannot be covered by very ample lattice $3$-simplices. 
\end{proposition}
\begin{proof}
Suppose that  $P$ can be covered by very ample $3$-simplices $P=\cup_{i=1}^nP_i$. Then we have each $P_i$ is normal and $k_{P_i}=1$ by \cite[Proposition 2.2]{Ogata2005}. Hence, $k_P= 1$ by Lemma \ref{inequalities on k_P}. This contradicts the result of \cite[Theorem 3.3]{Beck2015} that $k_P=s-1\ge 3$. Therefore, $P$ cannot be covered by very ample simplices.
\end{proof}
From Proposition \ref{3d non-normal cannot be covered by very ample simplices} it follows that the polytope $P$ defined in Example \ref{Bruns example} cannot be covered by very ample lattice $3$-simplices.

\begin{example}
For $d\ge 3$ and $h\ge 1$, Higashitani contructed a class of $d$-dimensional very ample lattice polytopes $\mathcal{P}_{d,h}$ with exactly $h$ holes, the lattice points in $kP$ that cannot be expressed at a sum of $k$ lattice points in $P$ $k$ runs from $2$ to $k_P-1$ (\cite[Theorem 1.]{Higashitani2014}), as follows. Let
\[u_i=\begin{cases}
0,&i=1\\
e_d,&i=2\\
e_2+\cdots+e_{d-1}, &i=3\\
h(e_2+\cdots+e_{d-1}+e_d),&i=4\\
(h-1)(e_2+\cdots+e_{d-1})+he_d, &i=5\\
h(e_2+\cdots+e_{d-1})+(h-1)e_d,&i=6\\
e_1+4e_d,&i=7\\
e_1+5e_d, &i=8\\
e_1+e_2+\cdots+e_{d-1}, &i=9\\
e_1+e_2+\cdots+e_{d-1}+e_d, &i=10,
\end{cases}\]
and
\begin{align*}
&v_i=e_i, &&i=2,\cdots,d-1\\
&v_i'=e_i+e_d, &&i=2,\cdots,d-1,
\end{align*}
where $e_1,\cdots, e_d$ are the unit coordinate vectors of $\RR^d$. Then define $\mathcal{P}_{d,h}$ to be the convex hull of
\[\{u_1,\cdots,u_{10}\}\cup\{v_i,v_i'\vert i=2,\cdots,d-1\}.\]
We have $\mathcal{P}_{d,h}$ is very ample and $k_{\mathcal{P}_{h,d}}=3$ (\cite[Theorem 1]{Higashitani2014}). Furthermore, $\mathcal{P}_{h,d}$ has $h$ holes and its facets are all normal (\cite[Lemma 5 \& 6]{Higashitani2014}), so the holes are interior lattice points of $\mathcal{P}_{h,d}$. Thus $\deg(\mathcal{P}_{h,d})=\dim(\mathcal{P}_{h,d})=d\ge 3=k_{\mathcal{P}_{h,d}}$.

By Proposition \ref{EG for deg P>kP}, we have that $\reg(X)= \deg(\mathcal{P}_{h,d})+1$, where $X\subseteq\PP^r$ is the toric variety obtained from $\mathcal{P}_{d,h}$. The Eisenbud-Goto conjecture holds for $X$ because of \cite[Proposition 2.2]{Hofscheier2017}.

We have in this case $m_{\mathcal{P}_{h,d}}=k_{\mathcal{P}_{h,d}}=3$, $d_{\mathcal{P}_{h,d}}=2$, so Theorem \ref{New bound of k-normality} yields
\[k_{\mathcal{P}_{h,d}}\le (m_P-d_P-1)n+\nu_P+1=\nu_P+1\le n.\]
This is much stronger than the Sturmfels' bound (\cite{Sturmfels1995}): $k_P\le |P\cap M|\cdot \Vol(P)\cdot(|P\cap M|-d-1)-1$. 
\end{example}
\section{On $d_P$ and Normal Polytopes}
We will give a short survey on $d_P$ in this section. We first begin with some upper bounds of $d_P$: 
%

\begin{proposition}\label{d_P <=deg P}
Let $P$ be a lattice polytope. Then 
\begin{enumerate}
\item \cite[Proposition IV.10]{Hering2006} If $P$ is not a standard simplex, then
\[d_P\le \deg P.\]
\item \cite[Proposition 2.2]{Hofscheier2017} If $P$ is spanning, in particular if $P$ is very ample, then
\[d_P\le \Vol(P)+d+1-|P\cap M|.\]
\end{enumerate}
\end{proposition}
\begin{remark}
Unfortunately, unlike Lemma \ref{d_P <= m_P <= k_P}, $P$ is normal does not implies that $d_P=\deg(P)$, with the standard simplex is a counterexample. In this case, we have $d_P=1$ while $\deg(P)=0$. In addition, it is not the case that $d_P=\deg(P)$ implies $P$ is normal, with a counterexample given in Example \ref{Bruns example}, where $d_P=\deg(P)=2$.
\end{remark}

Since $d_P=1$ if and only if $P$ is normal, we obtain a simple combinatorial proof for part of \cite[Proposition 6.9]{Blekherman2015}.

\begin{corollary}[{\cite[Proposition 6.9]{Blekherman2015}}]\label{deg=1 implies normality}
	Any lattice polytope of degree $0$ or $1$ is normal.
\end{corollary}

\begin{proof}
If $\deg P=0$, then $P$ is a basic simplex, so it is normal. Now suppose that $\deg P=1$. By Proposition \ref{d_P <=deg P}, $d_P\le \deg P\le 1$, which implies that $P$ is normal.
\end{proof}

By Proposition \ref{kP=reg(X)+1}, the Eisenbud-Goto conjecture would imply the following inequalities, that we summarize in the follwing question.
\begin{question}\label{Does EG holds for toric varieties?}
Let $P$ be a lattice polytope, is it the case that $d_P\le \Vol(P)+d+1-|P\cap M|$ even if $P$ is not a spanning polytope? If it is not then is it true that $d_P\le \Vol(P)$ for all lattice polytope $P$? Also, if $P$ is very ample lattice polytope, do we always have $k_P\le \Vol P$? 
\end{question}
It is obviously true that $d_P\le \Vol(P)$ if $\Vol(P)\ge d-1$. The interesting case is when $\Vol(P)\le d-2$; i.e, $P$ is a ``small" polytope.
\begin{example}
Let $P$ be a lattice polytope with $\dim P\le 3$. Then 
\[d_P\le \Vol P.\]
Also, for any $d$-dimensional lattice polytope $P$, if $d_P\le 2$, then $d_P\le \Vol(P)$.
\end{example}
\begin{proof}
In dimension $2$, since $d_P=1$ always, we have $d_P\le \Vol(P)$. If $\dim P=3$, then either $d_P=1$, so $d_P\le \Vol(P)$ trivially, or $d_P=2$. If $d_P=2$, then $\Vol(P)\ge d_P$. This is because if $\Vol(P)=1$ then $P$ is a standard simplex and so $P$ would be normal and $d_P=1$, a contradiction. Hence, $d_P\le \Vol(P)$ for the case $\dim P=3$ as well.
\end{proof}

In general, the difference $k_P-d_P$ cannot be bounded by any polynomial of $\dim(P)$, with an example again given by Example \ref{Bruns example}.
\begin{proposition}
For any non-negative integer $n$, there exists a $3$-dimensional very ample lattice polytope $P$ such that $k_P-d_P=n$.
\end{proposition}
\begin{proof}
For $n=0$, any normal polytope $P$ would give the desired result. For $n\ge 1$, take $P$ to be the polytope in Example \ref{Bruns example} with $s=n+3$. Then we have $d_P=2$ and by \cite[Theorem 3.3]{Beck2015}, $k_P=n+2$. The conclusion follows.
\end{proof}
However, for very ample lattice simplices we have the following result.
\begin{proposition}[{\cite[Theorem 1.3.3 (a)]{Bruns1997}, \cite[Proposition 2.4]{Ogata2005}}]
Let $P$ be a very ample lattice simplex. Then 
\[k_P-d_P\le \dim P-1.\]
\end{proposition}
\begin{proof}
If $\dim P=2$ then $P$ is normal by \cite[Theorem 1.3.3 (a)]{Bruns1997}, so $0=k_P-d_P\le \dim P-1=1$. The case $\dim P\ge 3$ follows from the proof of \cite[Proposition 2.4]{Ogata2005}.
\end{proof}
The ultimate goal in bounding $k_P-d_P$ is to prove or disprove that for smooth lattice polytope we have $k_P-d_P=0$. This is another interpretation of Oda's question \cite{Oda1988}.

Another question is what dilations of $P$ are normal. It is well-known that if $P$ is a $d$-dimensional lattice polytope then $(d-1)P$ is normal (\cite{Liu1993}, \cite{Ewald1991}, \cite{Bruns1997}). The following lemma, which follows easily from the definition of $d_P$, gives a slightly improved result, which also implies \cite[Proposition IV.10]{Hering2006} because of Proposition \ref{d_P <=deg P}. 

\begin{proposition}\label{d_PP is normal}
Let $P$ be a lattice polytope. Then $mP$ is normal for every $m\ge d_P$.
\end{proposition}

As a corollary of Proposition \ref{d_P <=deg P} and Proposition \ref{d_PP is normal}, we have:
\begin{corollary}
For any very ample lattice polytope $P$, $\Vol(P)\cdot P$ is normal.
\end{corollary}

%
%

Furthermore, $d_P$ is a possible candidate for the minimum number $n_P$ with the property that $kP$ is normal for every $k\ge n_P$.
\begin{question}\label{minimality of d_P?}
Let $P$ be a lattice polytope. Then is it always the case that
\[d_P=\min\{n\in\ZZ_{\ge 1}\mid kP\text{ is normal for all }k\ge n\}?\]
\end{question}
This is true in case $\dim P=2$ or $\dim P=3$.
\begin{example}
Let $P$ be a lattice polytope such that $\dim(P)\le 3$. Then 
\[d_P=\min\{n\in\ZZ_{\ge 1}\mid kP\text{ is normal for all }k\ge n\}\]
\end{example}
\begin{proof}
The case $\dim P=2$ is trivial because any polygon is normal by \cite[Theorem 1.3.3 (a)]{Bruns1997}. If $\dim P=3$, either $d_P=1$ or $d_P=2$. For the case $d_P=1$, $P$ is then normal; hence, the statement is true in this case. If $d_P=2$, we have  $P$ is not normal but $kP$ is normal for all $k\ge 2$ by Proposition \ref{d_PP is normal}. The conclusion follows.
\end{proof}

However, it is possible that $kP$ is normal for some $k< d_P$; i.e., $nP$ is normal does not imply $(n+1)P$ is also normal. Laso\'{n} and Micha\l{}ek (\cite{Lason2017}) found an example of a polytope $P$ such that $2P$ and $3P$ are normal but $5P$ is not. For such a $P$ we must have $5<d_P$; otherwise, $5P$ must be normal by Proposition \ref{d_PP is normal}. This example also implies that the inequalities $k_{P+Q}\le \max\{k_P,k_Q\}$ and $d_{P+Q}\le \max\{d_P,d_Q\}$ do not hold in general. On the other hand, it is easy to see that for any lattice polytopes $P$ and $Q$, we have 
\[ d_{P\cup Q}\le d_{P\times Q}=d_{P\odot Q}= \max\{d_P,d_Q\},\]
and if $P$ and $Q$ are very ample,
\[k_{P\cup Q}\le k_{P\times Q}= k_{P\odot Q}=\max\{k_P,k_Q\},\]
where $P\times Q$ and $P\odot Q$ are the Cartesian product and the join of $P$ and $Q$, respectively. 

\bigskip
The following corollary gives  criterion for normality of polytopes with small degree.
 
\begin{corollary}
	Let $P$ be a lattice polytope. Suppose that $P$ is $k$-normal for all $k\le \dim P-\deg P$. If $2\deg P\le \dim P$ then $P$ is normal.
\end{corollary}
\begin{proof}
	Since $d_P\le \deg(P)$ and $2\deg P\le \dim P$, we have $\dim P-\deg P \ge d_P$. Thus, $P$ is normal because of the definition of $d_P$.
\end{proof}

Since $P$ is $k$-normal does not imply $P$ is also $(k+1)$-normal in general (\cite[Theorem 9]{Handelman1990}), it makes sense to verify $k$-normality for $k\le \dim P-\deg P$ in the above corollary. 

\bibliographystyle{amsalpha}
\bibliography{Bibiliography}{}
\bibliographystyle{plain}

\end{document}